\documentclass[11pt,reqno]{amsart}
\usepackage{amssymb,latexsym}
\usepackage{graphicx}
\usepackage{fancyhdr,amssymb}
\usepackage{url}		%does nice formatting of URLs

\theoremstyle{plain}
\numberwithin{equation}{section}
\newtheorem{thm}{Theorem}[section]
\newtheorem{theorem}[thm]{Theorem}
\newtheorem{lemma}[thm]{Lemma}
\newtheorem{corollary}[thm]{Corollary}

\begin{document}
\fancyhead{}
\renewcommand{\headrulewidth}{0pt}
\fancyfoot{}
\fancyfoot[LE,RO]{\medskip \thepage}
\fancyfoot[LO]{\medskip MONTH YEAR}
\fancyfoot[RE]{\medskip VOLUME , NUMBER }

\setcounter{page}{1}

\title[Subsequences and Divisibility by Powers of the Fibonacci Numbers]{Subsequences and Divisibility by Powers of the Fibonacci Numbers}
\author{Kritkhajohn Onphaeng} 
\address{Department of Mathematics and Statistics, Faculty of Science\\
                Prince of Songkla University\\
                Songkla\\
               91000, Thailand}
\email{dome3579@gmail.com}
\author{Prapanpong Pongsriiam} 
\address{Department of Mathematics, Faculty of Science\\
                Silpakorn University\\
                Nakornpathom\\
               73000, Thailand}
\email{prapanpong@gmail.com}
\thanks{The second author receives financial support from Faculty of Science, Silpakorn University, Thailand, contract number RGP 2555-07. We would like to thank C. Panraksa, A. Tangboonduangjit, and K. Wiboonton for giving us their documents \cite{TW1, TW} which lead to the publication of this article. We also wish to thank the anonymous referee for his or her suggestions which improve the presentation of this article.\\
Correspondence should be addressed to Prapanpong Pongsriiam: prapanpong@gmail.com}

\begin{abstract}
Let $F_n$ be the $n$th Fibonacci number. Let $m, n$ be positive integers. Define a sequence $(G(k,n,m))_{k\geq 1}$ by $G(1,n,m) = F^m_n$, and $G(k+1,n,m) = F_{nG(k,n,m)}$ for all $k\geq 1$. We show that $F_n^{k+m-1}\mid G(k,n,m)$ for all $k, m, n\in\mathbb N$. Then we calculate $\frac{G(k,n,m)}{F_n^{k+m-1}}\pmod{F_n}$.
\end{abstract}

\maketitle

\section{Introduction}
The Fibonacci sequence is defined by the recurrence relation $F_1=F_2=1$ and $F_n=F_{n-1}+F_{n-2}$ for $n\geq3$. These numbers are famous for possessing wonderful properties, see \cite{BQ, KM, Ko, Va} for additional references and history. The Fibonacci Association was formed in order to provide enthusiasts an opportunity to exchange ideas about Fibonacci numbers and related topics. In this article, we would like to share an idea on the divisibility properties of the sequence $(G(k,n,m))_{k\geq 1}$ defined as follows: for each $n, m\geq 1$, let 
$$
G(1,n,m) = F^m_n\;\text{and let}\;G(k+1,n,m) = F_{nG(k,n,m)}\;\text{for all $k\geq 1$}.
$$ 
For example, the first few terms of this sequence are $F_n^m$, $F(nF_n^m)$, $F(nF(nF_n^m))$, where we write $F(\ell)$ instead of $F_\ell$ for convenience. So $(G(k,n,m))_{k\geq 2}$ (neglecting the term $k=1$) is a subsequence of Fibonacci numbers. Our work is motivated by the results in \cite{TW1, TW}. Recall that for integers $d\geq 2$, $k\geq 0$, and $a\geq 1$ we say that $d^k$ exactly divides $a$ and write $d^k\parallel a$ if $d^k\mid a$ and $d^{k+1}\nmid a$. Tangboonduangjit and Wiboonton \cite{TW} show that $F_n^k\mid G(k,n,1)$ for all $n, k\in\mathbb N$. Then Panraksa, Tangboonduangjit, and Wiboonton \cite{TW1} show that $F_n^k\parallel G(k,n,1)$ for all $n\geq 4$, $k\geq 1$. They \cite{TW1} also calculate $\frac{G(k,n,1)}{F_n^k}\pmod{F_n}$ for $k=2$ and $k=3$. In this article, we extend their result \cite{TW1, TW} by showing that $F_n^{k+m-1}\mid G(k,n,m)$ for every $k, m, n\in\mathbb N$. We also calculate $\frac{G(k,n,m)}{F_n^{k+m-1}}\pmod{F_n}$ for all $k\geq 2$ and $m\geq 1$. The precise statements are as follows:
\begin{theorem}\label{mainthm}
The following statements hold:
\begin{itemize}
\item[(i)] $F_n^{k+m-1}\mid G(k,n,m)$ for every $k, m, n \geq 1$.
\item[(ii)] Let $k\geq2$ and $n\geq 3$. Then 
$$
\frac{G(k,n,m)}{F_n^{k+m-1}} \equiv
 \begin{cases}
 1, &\text{if $2\mid k$ and $3\nmid n$ or if $2\nmid k$, $3\nmid n$, and $4\nmid n$};\\
 F_{n-1}, &\text{if $2\nmid k$, $3\nmid n$ and, $4\mid n$};\\
 \left(\frac{F_{n-3}}{2}\right)^{k-1}, &\text{if $2\nmid k$, $3\mid n$, and $m\geq 2$ or} \\
&\quad \text{if $2\mid k$, $3\mid n$ and $m=1$};\\
 (-1)^n\left(\frac{F_{n-3}}{2}\right)^{k-1}, &\text{if $2\mid k$, $3\mid n$, and $m\geq 2$ or}\\
&\quad \text{if $2\nmid k$, $3\mid n$, and $m=1$},
 \end{cases}
$$
 where the congruence is taken modulo $F_n$. 
\end{itemize}
\end{theorem}
Substituting $k=2, 3$ and $m=1$ in Theorem \ref{mainthm}, we obtain the results in \cite{TW}. With a bit more work, we obtain $F_n^{k+m-1}\parallel G(k,n,m)$ for every $k\geq 2$, $n\geq 4$, and $m\geq 1$, which also extends the results in \cite{TW1, TW}. Let us record it in the next Corollary.
\begin{corollary}\label{maincor}
The following statements hold:
\begin{itemize}
\item[(i)] $F_n^{k+m-1}\parallel G(k,n,m)$ for every $k\geq 2$, $n \geq 4$, $m\geq 1$.
\item[(ii)] For $n\geq 3$, $\frac{G(2,n,1)}{F_n^2} = \frac{F_{nF_n}}{F_n^2} \equiv \begin{cases}
1\pmod{F_n}, &\text{if $3\nmid n$;}\\
\frac{1}{2}F_{n-3}\pmod{F_n}, &\text{if $3\mid n$}.
\end{cases}$
\item[(iii)] For $n\geq 3$, $\frac{G(3,n,1)}{F_n^3} = \frac{F\left(nF_{nF_n}\right)}{F_n^3} \equiv \begin{cases}
1\pmod{F_n}, &\text{if $3\nmid n$ and $4\nmid n$;}\\
F_{n-1}\pmod{F_n}, &\text{if $3\nmid n$ and $4\mid n$;}\\
(-1)^n\left(\frac{F_{n-3}}{2}\right)^2\pmod{F_n}, &\text{if $3\mid n$}.
\end{cases}$
\end{itemize}
\end{corollary}
To prove Theorem \ref{mainthm}, we need a number of well known identities listed in the next section.
%%%%%%%%%%%%%%%%%%%%%%%%%%%%%%%%%%%%%%%%%%%%%%%%
%%%%%%%%%%%%%%%%%%%%%%%%%%%%%%%%%%%%%%%%%%%%%%%%
\section{Preliminaries and Lemmas}
Let $m, n$, and $r$ be positive integers. The following results are well known and will be used throughout this article :
\begin{equation}\label{pre1}
\gcd(F_m,F_n)=F_{\gcd(m,n)}.
\end{equation}
\begin{equation}\label{pre2}
F_{nr} = \sum_{j=1}^{r}{r\choose j}F_n^jF_{n-1}^{r-j}F_j.
\end{equation}
\begin{equation}\label{pre3}
\text{If $m\geq3$, then $F_m\mid F_n$ if and only if $m\mid n$}.
\end{equation}
\begin{equation}\label{pre3.5}
F_{n+1}F_{n-1}-F_n^2 = (-1)^n.
\end{equation}
\begin{equation}\label{pre4}
F_{n-1}^{2} \equiv F_{n+1}^{2} \equiv (-1)^{n} \pmod {F_n}.
\end{equation}
\begin{equation}\label{pre9}
F_{m+n} = F_{m+1}F_n + F_mF_{n-1}.
\end{equation}
\begin{equation}\label{pre5}
\text{If $m\equiv n\pmod 6$, then $F_m\equiv F_n\pmod 4$}.
\end{equation}
 For the reader's convenience, let us give references for the above identities. The relation (\ref{pre1}) and (\ref{pre3}) can be found in \cite[p. 197--198]{Ko}, (\ref{pre3.5}) is the Cassini's identity \cite[p. 74]{Ko}, (\ref{pre4}) follows from Cassini's identity and the fact that $F_{n+1}\equiv F_{n-1}\pmod {F_n}$. The identity (\ref{pre2}) is an important tool in proving our theorem and can be found in several articles such as \cite{BR, Ha, Ho}. The identity (\ref{pre9}) is proved in \cite[p. 294]{GK}. For the relation (\ref{pre5}), noticing that $F_0\equiv F_6$ and $F_1\equiv F_7\pmod 4$, we see from the recursive definition of the Fibonacci sequence that $(F_n)\mod 4$ has period 6. Therefore (\ref{pre5}) holds. The following is a consequence of (\ref{pre5}) which will be used repeatedly.
\begin{equation}\label{pre6}
\text{If}\;3\nmid n\; \text{and}\; 2\nmid n,\;\text{then}\;F_n\equiv 1\pmod 4.
\end{equation}
The proof of (\ref{pre6}) is as follows. Assume that $3\nmid n$ and $2\nmid n$. Then $n\equiv 1,5 \pmod 6$. Therefore $F_n\equiv F_1$, $F_5 \pmod 4$ by (\ref{pre5}). Since $F_1 = 1 \equiv 1$, $F_5 = 5 \equiv 1 \pmod 4$, we have $F_n \equiv 1 \pmod 4$. We refer the reader to \cite{BQ, KM, Ko, Le, R, Ro, Va, Vi} for more details and references. The next lemma is an important tool to prove our results.
\begin{lemma}\label{lemmahoggatt1}
\cite{BR, Ho} For each positive integer $n, r$, and $s$, if $F_n^{s-1}\mid r$, then $F_n^s\mid F_{nr}$.
\end{lemma}
Note that Lemma \ref{lemmahoggatt1} was first proved in 1977 by V. E. Hoggatt, Jr., and Marjorie Bicknell-Johnson \cite{Ho} but it or its consequence has been proved again recently \cite{BR, Ma, ST} by a different method. It is worth noting that many authors \cite{BR, Ma, ST} do not seem to realize the existence of Lemma \ref{lemmahoggatt1} which was first proved in \cite{Ho}. For example, since $F_n^k\mid F_n^k$, we obtain from Lemma \ref{lemmahoggatt1} immediately that 
\begin{equation}\label{lostthmeq1}
F_n^{k+1}\mid F_{nF_n^k}.
\end{equation}
It was mentioned by D. Marques \cite[p. 241]{Ma} that to the best of his knowledge, (\ref{lostthmeq1}) was first proved by Benjamin and Rouse \cite[2006]{BR}, using a combinatorial approach, and a second proof of (\ref{lostthmeq1}) is due to Seibert and Trojovsky \cite[2008]{ST} by using mathematical induction together with an identity for $\frac{F_{nm}}{F_m}$. D. Marques \cite[2012]{Ma} himself also gave another proof of (\ref{lostthmeq1}) by applying Lengyel's Theorem \cite{Le} on the $p$-adic order of Fibonacci and Lucas numbers. So it is nice to bring Lemma \ref{lemmahoggatt1} back to the literature.
%%%%%%%%%%%%%%%%%%%%%%%%%%%%5
The following lemmas are of our own. We also prove Lemma \ref{newlemma1} and Lemma \ref{newlemma2} in \cite{Po} but our manuscript has not been formally published in a journal. So we give a proof here for completeness.
\begin{lemma}\label{newlemma1}
Let $a, j\geq 1$ and $s\geq 2$ be positive integers. Assume that $c$ is the smallest nonnegative integer such that $j\mid as^c$. Then there exists a prime $p$ such that $p\mid s$ and $p^c\mid j$.
\end{lemma}
\begin{proof}
If $c=0$, then we can choose any prime $p$ that divides $s$. So we assume that $c\geq 1$. Since $j\mid as^c$, there is $q\in\mathbb N$ such that $as^c = jq$. If $s\mid q$, then $as^{c-1} = j\left(\frac{q}{s}\right)$ and therefore $j\mid as^{c-1}$ which contradicts the minimality of $c$. Thus $s\nmid q$. Then there exists a prime $p$ such that $p^\alpha\parallel s$, $p^{\beta}\parallel q$, and $\alpha > \beta \geq 0$. Since $q\mid as^c$, we see that
\begin{equation}\label{inprove1}
\frac{q}{p^{\beta}}\mid a\left(\frac{s}{p^{\alpha}}\right)^c.
\end{equation}
Let $n = \frac{a\left(\frac{s}{p^\alpha}\right)^c}{\frac{q}{p^\beta}}\cdot p^{(\alpha-1)c-\beta}$. Then $p^cn = \frac{as^c}{q} = j$. Now by (\ref{inprove1}) and the fact that $\alpha > \beta$, we see that $n\in\mathbb{N}$, $p^c\mid j$, and $p\mid s$.
\end{proof}
%%%%%%%%%%%%%%%%%%%%%%%%%%%%%%%%%
\begin{lemma}\label{newlemma2}
Let $k, \ell, r, s \in\mathbb{N}$ and $s^k\mid r$. Then $s^{k+\ell}\mid {r\choose j}s^j $ for all $1\leq j\leq r$ satisfying $2^{j-\ell+1} > j$. In particular, $s^{k+2}\mid {r \choose j}s^j$ for all $3 \leq j \leq r$.
\end{lemma}
\begin{proof}
Note that the result holds trivially when $s=1$ or $j\geq k+\ell$. So we assume that $s\geq 2$ and $1\leq j < k+\ell$ satisfying $2^{j-\ell+1} > j$. Since $s^k\mid r$, there is $b\in\mathbb{N}$ such that $r = s^kb$. Then
$$
{r\choose j} = \frac{r}{j}{r-1\choose j-1} = \frac{s^kb}{j}{r-1\choose j-1} = s^{k-j+\ell}\cdot\frac{s^{j-\ell}b{r-1\choose j-1}}{j}.
$$
Suppose for a contradiction that $j \nmid  s^{j-\ell}b{r-1\choose j-1}$. Since $j\mid s^kb{r-1\choose j-1}$, there exists the smallest nonnegative integer $c$ such that $j\mid s^cb{r-1\choose k-1}$. Then
\begin{equation}\label{inprove2}
j-\ell+1 \leq c \leq k.
\end{equation}
By Lemma \ref{newlemma1}, there exists a prime $p$ such that $p^c\mid j$. Now $j \geq p^c \geq 2^c \geq 2^{j-\ell+1} > j$, a contradiction. Hence $j \mid s^{j-\ell}b{r-1\choose j-1}$ and therefore $s^{k-j+\ell}\mid {r\choose j}$ as required.
\end{proof}
%%%%%%%%%%%%%%%%%%%%%%%%%%%%%%%%%
%%%%%%%%%%%%%%%%%%%%%%%%%%%%%%%%%%
\begin{lemma}\label{newlemma4}
Let $m, n \geq 1$, $k \geq 2$, and $r = G(k,n,m).$ Then the following statements hold.
\begin{itemize}
\item[(i)] $2 \mid r$ if and only if $3\mid n$ or $4\mid n$.
\item[(ii)] If $2\nmid n$ and $3 \nmid n$, then $r \equiv 1\pmod 4$.
\item[(iii)] If $3\mid n$, then $r \equiv 0 \pmod 8$.
\end{itemize}
\end{lemma}
\begin{proof}
Throughout the proof, we will apply (\ref{pre3}) repeatedly for the case $m =3, 4$ so let us record it here again. 
$$
\text{$3\mid n$ if and only if $2\mid F_n$ and $4\mid n$ if and only if $3\mid F_n$.}
$$
We will prove (i) by induction on $k$. Consider the equivalence: 
\begin{align*}
2\mid G(2,n,m) &\Leftrightarrow F_3\mid F_{nF_n^m} \Leftrightarrow 3\mid nF_n^m \Leftrightarrow 3\mid n \vee 3\mid F_n\\
& \Leftrightarrow 3\mid n \vee F_4\mid F_n \Leftrightarrow 3\mid n \vee 4\mid n.
\end{align*}
This proves (i) when $k=2$. Next, we consider the case $k=3$.
\begin{align}\label{inprove3}
2\mid G(3,n,m) &\Leftrightarrow 3\mid nG(2,n,m) \Leftrightarrow 3\mid n \;\vee\; 3\mid G(2,n,m)\notag\\
& \Leftrightarrow 3\mid n \;\vee\; 4\mid nF_n^m.
\end{align}
If $4\mid nF_n^m$ and $3\nmid n$, then $(2,F_n) = 1$ which implies that $4\mid n$. So (\ref{inprove3}) is equivalent to the condition $3\mid n$ or $4\mid n$. This proves (i) for $k=3$. Now assume that $k \geq 3$ and (i) holds for $2, 3, \ldots, k$. Then
\begin{align}\label{inprove13.5}
2\mid G(k+1,n,m) &\Leftrightarrow 3\mid nG(k,n,m)\notag\\
&\Leftrightarrow 3\mid n \;\vee\; 3\mid G(k,n,m)\notag\\
&\Leftrightarrow 3\mid n \;\vee\; 4\mid nG(k-1,n,m).
\end{align}
If $2\mid G(k-1,n,m)$, we can apply the induction hypothesis to conclude that $3\mid n$ or $4\mid n$. If $2\nmid G(k-1,n,m)$, then $\gcd(4,G(k-1,n,m)) = 1$, and therefore $4\mid n$. So (\ref{inprove13.5}) is equivalent to $3\mid n$ or $4\mid n$. This shows that (i) holds for all $k \geq 2$. \\
For (ii), we assume that $2\nmid n$ and $3\nmid n$. Again we will prove that $r\equiv 1 \pmod 4$ by induction on $k$. Similar to the proof of (i), we see that $3\mid nF_n^m\Leftrightarrow 3\mid n$ or $4\mid n$ and $2\mid nF_n^m \Leftrightarrow 2\mid n$ or $3\mid n$. Since $2\nmid n$ and $3\nmid n$, we see that $3\nmid nF_n^m$ and $2\nmid nF_n^m$.
By (\ref{pre6}), $G(2,n,m) = F_{nF_n^m} \equiv 1\pmod 4$. The case $k=3$ can be obtained similarly. Next assume that $k\geq3$ and (ii) holds for $2, 3, \ldots, k$. Then
\begin{align}\label{inprove4}
3\mid nG(k,n,m)& \Leftrightarrow 3\mid n \;\vee\; 4\mid nG(k-1,n,m)\notag\\
&\Leftrightarrow 3\mid n \;\vee\; 4\mid n.
\end{align}
where the penultimate equivalence is done by applying (i). From (\ref{inprove4}) and the fact that $3\nmid n$ and $2\nmid n$, we see that $3\nmid nG(k,n,m)$. We also have by the induction hypothesis that $2\nmid G(k,n,m)$ and therefore $2\nmid nG(k,n,m)$. Since $3\nmid nG(k,n,m)$ and $2\nmid nG(k,n,m)$, we obtain by (\ref{pre6}) that $G(k+1,n,m) = F_{nG(k,n,m)} \equiv 1 \pmod 4$. This proves (ii). The proof of (iii) is similar. Hence the proof is complete.
\end{proof}
%%%%%%%%
%%%%%%%%
%%%%%%%%
%%%%%%%%
%%%%%%%%
%%%%%%%%
%%%%%%%%
%%%%%%%%%%%%%%%%%%%%%%%%%%%%%%%%%%%%%%%%%%%%%%%%%%%%%%%%
%%%%%%%%%%%%%%%%%%%%%%%%%%%%%%%%%%%%%%%%%%%%%%%%%%%%%%%%%%%%
%%%%%%%%%%%%%%%%%%%%%%%%%%%%%%%%%%%%%%%%%%%%%%%%%%%%%%%%%%%%
%%%%%%%%%%%%%%%%%%%%%%%%%%%%%%%%%%%%%%%%%%%%%%%%%%%%%%
%%%%%%%%%%%%%%%%%%%%%%%%%%%%%%%%%%%%%%%%%%%%%%%%%%%%%%
\section{Proof of The Main Results}
\noindent \textbf{Proof of Theorem \ref{mainthm}(i)}.\\[0.1cm]
Since $G(1,n,m) = F_n^m$, the result holds for $k=1$. If $k\geq 1$ and $F_n^{k+m-1}\mid G(k,n,m)$, then we obtain by Lemma \ref{lemmahoggatt1} that $F_n^{k+1+m-1}$ divides $F_{nG(k,n,m)} = G(k+1,n,m)$. Therefore the result follows by induction on $k$.\qed\\[0.2cm]
%%%%%%%%%%%%%%%
%%%%%%%%%%%%%%%5
\noindent \textbf{Proof of Theorem \ref{mainthm}(ii)}.
\begin{itemize}
\item[\textbf{\underline{Case $m=1$}}]
Basis Step: We let $m=1$ and we will prove the statement by induction on $k$. The case $k=2$ is proved in \cite{TW1}. But we will prove it again here. Our proof is shorter and gives an idea on the proof of the induction step and the proof of the case $m\geq 2$. The statement we want to prove is as follows:
$$
\frac{F_{nF_n}}{F_n^2} \equiv
\begin{cases}
1\pmod{F_n},\quad&\text{if $3\nmid n$;}\\
\frac{1}{2}F_{n-3}\pmod{F_n},\quad&\text{if $3\mid n$}.
\end{cases}
$$
Let $r=F_n$, $A = F_{n-1}^{r-1}$, and $B = \frac{r(r-1)}{2}F_{n-1}^{r-2}$. Since $F_n\mid r$, we obtain by Lemma \ref{newlemma2} that $F_n^3\mid {r\choose j}F_n^j$ for all $3\leq j\leq r$. That is $F_n\mid {r\choose j}F_n^{j-2}$ for each $3\leq j\leq r$. Then by (\ref{pre2}), we have 
\begin{align*}
\frac{F_{nr}}{F_n^2} = \sum_{j=1}^r{r\choose j}F_n^{j-2}F_{n-1}^{r-j}F_j \equiv \sum_{j=1}^2{r\choose j}F_n^{j-2}F_{n-1}^{r-j}F_j \equiv A+B\pmod{F_n}.
\end{align*}
\begin{itemize}
\item [Case 1] Assume that $3\nmid n$. Then $n\equiv 1, 2, 4, 5\pmod 6$. If $n\equiv 1, 2, 5$, then $F_n\equiv F_1, F_2, F_5 \equiv 1\pmod 4$, by (\ref{pre5}). Therefore $F_n$ is odd and $\frac{r-1}{2} = \frac{F_n-1}{2}$ is even. This implies that 
$$
B = \left(\frac{r-1}{2}\right)F_nF_{n-1}^{r-2}\equiv 0\pmod{F_n},
$$
and 
$$
A = F_{n-1}^{r-1} = F_{n-1}^{2\left(\frac{r-1}{2}\right)} \equiv (-1)^{n\left(\frac{r-1}{2}\right)}\equiv 1\pmod{F_n},\;\text{by (\ref{pre4})}.
$$
Hence $\frac{F_{nr}}{F_n^2}\equiv A+B\equiv 1\pmod{F_n}$. In what follows, (\ref{pre4}) and (\ref{pre5}) may be applied without mentioning.
\item [Case 2] Assume that $3\mid n$. Then $n\equiv 0, 3\pmod 6$. If $n\equiv 0\pmod 6$, then $n$ is even, $F_n\equiv F_0\equiv 0\pmod 4$, and therefore 
$$
F_{n-1}^{r-2} \equiv (-1)^{n\left(\frac{r-2}{2}\right)} \equiv 1\pmod{F_n}.
$$
Similarly, if $n\equiv 3\pmod 6$, then $F_n \equiv F_3 \equiv 2\pmod 4$ so $\frac{r-2}{2} = \frac{F_n-2}{2}$ is even and $F_{n-1}^{r-2} \equiv 1\pmod{F_n}$. In any case, we have $F_{n-1}^{r-2}\equiv 1\pmod{F_n}$. Therefore
\begin{align*}
\frac{F_{nr}}{F_n^2} &\equiv A+B = F_{n-1}^{r-2}\left(F_{n-1}-\frac{1}{2}F_n+\frac{1}{2}F_n^2\right)\\
&= F_{n-1}^{r-2}\left(\frac{1}{2}F_{n-3}+\frac{1}{2}F_n^2\right) \equiv \frac{1}{2}F_{n-3}\pmod{F_n}.
\end{align*} 
\end{itemize}
This completes the proof of the basis step $k=2$. \\
\noindent Induction Step: Next let $k\geq 2$ and assume that the statement is true for $k$. Let $r=G(k,n,1)$. Consider
\begin{align*}
\frac{G(k+1,n,1)}{F_n^{k+1}} = \frac{F_{nr}}{F_n^{k+1}} = \sum_{j=1}^{r}{r \choose j}F_n^{j-k-1}F_{n-1}^{r-j}F_j.
 \end{align*}
Since $F_n^k\mid r$, we obtain by Lemma \ref{newlemma2} that $F_n^{k+2} \mid {r \choose j}F_n^j$ for each $3\leq j \leq r$. Therefore $F_n\mid {r\choose j}F_n^{j-k-1}$ for each $3\leq j\leq r$. So the above sum is congruent to $A+B$ modulo $F_n$ where $A = rF_n^{-k}F_{n-1}^{r-1}$, $B = {r\choose 2}F_n^{-k+1}F_{n-1}^{r-2}$. We have 6 cases to consider.
\begin{itemize}
\item [Case 1] Assume that $2\nmid k$, $3\nmid n$, and $4\nmid n$. By Lemma \ref{newlemma4}(i), $r$ is odd. Therefore 
$$
B = \left(\frac{r-1}{2}\right)\frac{r}{F_n^k}F_nF_{n-1}^{r-2}\equiv 0\pmod{F_n}.
$$
By the induction hypothesis, $\frac{r}{F_n^k} \equiv 1 \pmod{F_n}$. Therefore 
$$
A \equiv F_{n-1}^{r-1} \equiv \left(-1\right)^{n\left(\frac{r-1}{2}\right)}\pmod{F_n}.
$$
If $2\mid n,$ then $A \equiv 1\pmod{F_n}$. If $2\nmid n$, then by Lemma \ref{newlemma4}(ii), $r \equiv 1 \pmod 4$, which implies that $A \equiv 1 \pmod {F_n}$. In any case, $A\equiv 1\pmod {F_n}$ and hence $A+B \equiv1 \pmod {F_n}$.
\item [Case 2] Assume that $2\mid k$, $3\nmid n$, and $4\nmid n$. Then $r$ is odd, by Lemma \ref{newlemma4}(i). The calculation of $A+B \pmod{F_n}$ is the same as that in Case 1 and we obtain $A+B \equiv 1 \pmod{F_n}$.
\item [Case 3] Suppose that $2\nmid k$, $3\nmid n$, and $4\mid n$. Then $r$ is even by Lemma \ref{newlemma4}(i). We write $2(A+B)=\frac{r}{F_n^k}F_{n-1}^{r-2}(2F_{n-1}+(r-1)F_n)$. By the induction hypothesis, 
$$
\frac{r}{F_n^k}\equiv F_{n-1}\pmod{F_n}.
$$
Since $r$ and $n$ are even, we obtain $F_{n-1}^{r-2}\equiv(-1)^{n(\frac{r-2}{2})}\equiv 1 \pmod{F_n}$. Therefore
\begin{align*}
2(A+B)\equiv F_{n-1}(2F_{n-1}) \equiv 2F_{n-1}^2\equiv 2(-1)^n \equiv 2 \pmod{F_n}.
\end{align*}
Since $3\nmid n$, we see that $(2,F_n)=1$ and thus $A+B \equiv 1 \pmod{F_n}$.
\item [Case 4] Suppose that $2\mid k$, $3\nmid n$, and $4\mid n$. Then by Lemma \ref{newlemma4}(i), $r$ is even. The calculation of $A+B\pmod{F_n}$ is similar to the one in Case 3 except that, in this case, $\frac{r}{F_n^k}\equiv 1 \pmod{F_n}$. We obtain $A+B\equiv F_{n-1}\pmod{F_n}$.
\item [Case 5] Assume that $2\nmid k$ and $3\mid n$. Then $F_n$ is even and, by Lemma \ref{newlemma4}(iii), $r\equiv 0 \pmod 8$. We write 
$$
A+B=\frac{r}{F_n^k}F_{n-1}^{r-2}\left(F_{n-1}+\frac{r-1}{2}F_n\right).
$$
We have $\frac{r}{F_n^k}\equiv(-1)^n\left(\frac{F_{n-3}}{2}\right)^{k-1}\pmod{F_n}$, by the induction hypothesis. In addition $F_{n-1}^{r-2}\equiv(-1)^{n\left(\frac{r-2}{2}\right)} \equiv (-1)^n\pmod{F_n}$ and $F_{n-1}+\frac{r-1}{2}F_n = F_{n-1}-\frac{F_n}{2}+\frac{rF_n}{2} = \frac{1}{2}F_{n-3}+\frac{rF_n}{2}\equiv\frac{1}{2}F_{n-3}\pmod{F_n}$. Therefore $A+B \equiv \left(\frac{F_{n-3}}{2}\right)^k \pmod{F_n}$.
\item [Case 6] Suppose that $2\mid k$ and $3\mid n$. Then $F_n$ is even and, by Lemma \ref{newlemma4}(iii), $r\equiv 0\pmod 8$. Similar to Case 5, we obtain the following  
\begin{align*}
\frac{r}{F_n^k} &\equiv \left(\frac{F_{n-3}}{2}\right)^{k-1},\\
F_{n-1}^{r-2} &\equiv (-1)^{n\left(\frac{r-2}{2}\right)} \equiv (-1)^n\quad\text{and}\\
F_{n-1}+\frac{r-1}{2}F_n &\equiv \frac{1}{2}F_{n-3},
\end{align*}
where all congruences are taken modulo $F_n$. Therefore $A+B \equiv (-1)^n\left(\frac{F_{n-3}}{2}\right)^k \pmod{F_n}$.
\end{itemize}
Combining the results in Case 1 to Case 6, we obtain
$$
\frac{G(k+1,n,1)}{F_n^{k+1}} \equiv
 \begin{cases}
 1, &\text{if $2\nmid k$ and $3\nmid n$ or if $2\mid k$ and}\\
&\quad \text{$3\nmid n$, $4\nmid n$ (by Case 1, Case 2, and Case 3)},\\
 F_{n-1}, &\text{if $2\mid k$ and $3\nmid n$, $4\mid n$ (by Case 4)},\\
\left(\frac{F_{n-3}}{2}\right)^{k}, &\text{if $2\nmid k$ and $3\mid n$ (by Case 5)},\\
(-1)^n\left(\frac{F_{n-3}}{2}\right)^{k}, &\text{if $2\mid k$ and $3\mid n$ (by Case 6)}.
\end{cases}
$$
Since $2\mid k$ if and only if $2\nmid k+1$, we can write the above congruence in the following form: 
$$
\frac{G(k+1,n,1)}{F_n^{k+1}} \equiv
 \begin{cases}
 1, &\text{if $2\mid k+1$ and $3\nmid n$ or if $2\nmid k+1$ and}\\
&\quad \text{$3\nmid n$, $4\nmid n$},\\
 F_{n-1}, &\text{if $2\nmid k+1$ and $3\nmid n$, $4\mid n$},\\
\left(\frac{F_{n-3}}{2}\right)^{k}, &\text{if $2\mid k+1$ and $3\mid n$},\\
(-1)^n\left(\frac{F_{n-3}}{2}\right)^{k}, &\text{if $2\nmid k+1$ and $3\mid n$},
\end{cases}
$$
where the congruences are taken modulo $F_n$. This proves the induction step and hence the proof of the case $m=1$ and $k\geq 2$ is complete.
%%%%%%%%%%%%%%%%%%%%%%%%%
%%%%%%%%%%%%%%%%%%%%%%%%%%%%%%%%
%%%%%%%%%%%%%%%%%%%%%%%%%
%%%%%%%%%%%%%%%%%%%%%%%%%%%%%%%%
\item[\textbf{\underline{Case $m\geq 2$}}]
Basis Step: We let $m\geq 2$ and we will prove this statement by induction on $k$. First we will consider the case $k=2$. Similar to the proof of the Case $m=1$, we let $r = F_n^m$ and we obtain that $\frac{G(2,n,m)}{F_n^{m+1}}$ is congruent to $A+B$ modulo $F_n$, where $A = F_{n-1}^{r-1}$ and $B = \frac{1}{2}F_n(r-1)F_{n-1}^{r-2}$. If $3\nmid n$, then we can follow the same argument of the proof of Case $m=1$, and we obtain $A+B\equiv 1\pmod{F_n}$. Now assume that $3\mid n$. The process of calculation is still the same as the proof of the Case $m=1$. Only at this time, $m\geq 2$, so $F_{n-1}^{r-2} \equiv (-1)^n\pmod{F_n}$. Therefore 
$$
A+B = F_{n-1}^{r-2}\left(\frac{1}{2}F_{n-3}+\frac{1}{2}rF_n\right) \equiv \frac{(-1)^n}{2}F_{n-3}\pmod{F_n}.
$$
In conclusion, 
$$
\frac{G(2,n,m)}{F_n^{1+m}} \equiv
 \begin{cases}
 1\pmod{F_n}, &\text{if $3\nmid n$;} \\
 (-1)^n\left(\frac{F_{n-3}}{2}\right)\pmod{F_n}, &\text{if $3\mid n$.}
 \end{cases}
$$
This proves the basis step.\\
\noindent Induction Step: Next let $k\geq 2$ and assume that the statement is true for $k$. Let $r=G(k,n,m)$. Then following the same argument as in proof of the Case $m=1$, we see that $\frac{G(k+1,n,m)}{F_n^{k+m}}$ is congruent to $A+B$ modulo $F_n$, where $A = rF_n^{1-k-m}F_{n-1}^{r-1}$ and $B = {r\choose 2}F_n^{-k-m+2}F_{n-1}^{r-2}$. 
\begin{itemize}
\item [Case 1] If $2\mid k$, $3\nmid n$, and $4\nmid n$ or $2\nmid k$, $3\nmid n$, and $4\nmid n$, then we can follow the argument of Case 1 and Case 2 in the proof of the Case $m=1$. We obtain $A+B \equiv 1 \pmod{F_n}$.
\item [Case 2] If $2\nmid k$, $3\nmid n$, and $4\mid n$ or $2\mid k$, $3\nmid n$, and $4\mid n$, then the proof is similar to Case 3 in the proof of the Case $m=1$. We obtain
$$
A+B \equiv
\begin{cases}
1\pmod{F_n}, &\text{if $2\nmid k$, $3\nmid n$, and $4\mid n$};\\
F_{n-1}\pmod{F_n}, &\text{if $2\mid k$, $3\nmid n$, and $4\mid n$}.
\end{cases}
$$
\item [Case 3] If $2\mid k$ and $3\mid n$, then the proof is similar to Case 5 and we obtain $A+B \equiv \left(\frac{F_{n-3}}{2}\right)^k \pmod{F_n}$. If $2\nmid k$ and $3\mid n$, then the proof is similar to Case 6 and we have $A+B \equiv (-1)^n\left(\frac{F_{n-3}}{2}\right)^k \pmod{F_n}$. 
\end{itemize}
This completes the proof. \qed
\end{itemize}
%%%%%%%%%%%%%%%%%%5
%%%%%%%%%%%%%%%%%%555
\noindent \textbf{Proof of Corollary \ref{maincor}}.\\[0.2cm]
Let $k\geq 2$, $n\geq 4$, and $m\geq 1$ be positive integers. Since $1\not\equiv 0\pmod{F_n}$ and $F_{n-1}\not\equiv 0\pmod{F_n}$, we see that $F_n^{k+m-1}\parallel G(k,n,m)$ for every $n$ such that $3\nmid n$. So we assume that $3\mid n$. Suppose for a contradiction that
\begin{equation}\label{pfmaincoreq1}
F_n\mid \left(\frac{F_{n-3}}{2}\right)^{k-1}.
\end{equation}
Since $\gcd \left(F_n, F_{n-3}\right) = F_{\gcd \left(n, n-3\right)} = F_3 = 2$, we obtain $\gcd\left(\frac{F_n}{2},\frac{F_{n-3}}{2}\right)=1$.\\
 Then $\gcd\left(\frac{F_n}{2},\left(\frac{F_{n-3}}{2}\right)^{k-1}\right) = 1$. But from (\ref{pfmaincoreq1}), we have $\frac{F_n}{2} \mid \left(\frac{F_{n-3}}{2}\right)^{k-1}$. Therefore 
$$
1 = \gcd\left(\frac{F_n}{2},\left(\frac{F_{n-3}}{2}\right)^{k-1}\right) = \frac{F_n}{2},
$$
which implies $n=3$. This contradicts the fact that $n\geq 4$. Hence $F_n\nmid \left(\frac{F_{n-3}}{2}\right)^{k-1}$ and thus $F_n^{k+m-1}\parallel G(k,n,m)$. This proves (i). By substituting $k=2$ and $k=3$ in Theorem \ref{mainthm}, we immediately obtain (ii) and (iii). This completes the proof.\qed
%%%%%%%%%%%%%%%%%%%%%%%%%%%%%%%%%%%%%%%%%%%%%%%%%%%%%%%%%%
%%%%%%%%%%%%%%%%%%%%%%%%%%%%%%%%%%%%%%%%%%%%%%%%%%%%%%%%%%
%%%%%%%%%%%%%%%%%%%%%%%%%%%%%%%%%%%%%%%%%%%%%%%%%%%%%

\medskip

\noindent MSC2010: 11B39
\end{document}